\documentclass{amsart}
\usepackage{mathtools}
\usepackage{enumerate}
\usepackage{tikz}
\usetikzlibrary{automata,positioning}
\usepackage[pdftex]{hyperref}
\usepackage{mathrsfs}
\usepackage{amsfonts}
\usepackage{latexsym}
\usepackage{epsfig}
\usepackage{indentfirst}
\graphicspath{{figures/}}
\usepackage{amsmath}
\usepackage{amsthm}
\usepackage{graphicx}
\usepackage{amssymb}
\usepackage{amscd}
\usepackage{latexsym}
\usepackage{subfigure}
\usepackage{epstopdf}
\usepackage{color}
\DeclareGraphicsExtensions{.eps}
\input xy
\xyoption{all}

\renewcommand{\theequation}{\thesection.\arabic{equation}}
\theoremstyle{definition}
\newtheorem{thm}{\textbf{Theorem} }[section]
\newtheorem{conj}{\textbf{Conjecture}}
\newtheorem{lem}[thm]{\textbf{Lemma}}

\newtheorem{prop}[thm]{\textbf{Proposition}}
\newtheorem*{rmk}{\textbf{Remark}}
\newtheorem{defn}[thm]{\textbf{Definition}}

\numberwithin{equation}{section} \makeatletter

\begin{document}

\title{On the Alexander polynomial and the signature invariant of two-bridge knots}

\author{Wenzhao Chen}
\address{Department of Mathematics, Michigan State University, 619 Red Cedar Road, C535 Wells Hall, East Lansing, MI 48824}
\email{chenwenz@math.msu.edu}
\thanks{}



\begin{abstract}
Fox conjectured the Alexander polynomial of an alternating knot is trapezoidal, i.e.\ the coefficients first increase, then stabilize and finally decrease in a symmetric way. Recently, Hirasawa and Murasugi further conjectured a relation between the number of the stable coefficients in the Alexander polynomial and the signature invariant. In this paper we prove the Hirasawa-Murasugi conjecture for two-bridge knots. 
\end{abstract}

\maketitle
\section{Introduction}
A knot is said to be alternating if it admits a diagram in which the crossings alternate between over- and underpasses. In 1962, Fox posed the following conjecture concerning a curious behavior of the Alexander polynomial of an alternating knot.
\begin{conj}[\cite{Fox62}]\label{Foxconj}
Let $K$ be an alternating knot with the Alexander polynomial $\Delta_K(t)=\Sigma_{j=0}^{2n}(-1)^ja_jt^{2n-j}$, $a_j>0$. Then
$$a_0<a_1<\cdots<a_{n-m-1}<a_{n-m}=\cdots=a_{n+m}>a_{n+m+1}>\cdots>a_{2n}.$$
\end{conj}
Polynomials satisfying the above condition are called trapezoidal, so this conjecture is known as Fox's trapezoidal conjecture. This conjecture remains open.It is, however, supported by the verification on several classes of alternating knots. The case of two-bridge knots is confirmed by Hartley \cite{Har79}. More generally Murasugi proved it for alternating algebraic knots \cite{Mur85}. The case of genus two alternating knots has also been verified by Ozsv\'ath and Szab\'o using Heegaard Floer homology \cite{OS03}, and by Jong via a combinatorial method \cite{Jon09}. Recently, Hirasawa and Murasugi showed that the conjecture holds for alternating stable knots \cite{HM13}. Moreover, in this case they observed that the signature of such knots are zero, and $m=0$ in Conjecture \ref{Foxconj}. Therefore, this progress led them to pose the following strengthened conjecture.

\begin{conj}[\cite{HM13}]\label{HMconjecture}
Let $K$ be an alternating knot, whose signature $|\sigma(K)|=2k$ and the Alexander polynomial $\Delta_K(t)=\Sigma_{j=0}^{2n}(-1)^ja_jt^{2n-j}$, $a_j>0$. Then
$$a_0<a_1<\cdots<a_{n-m-1}<a_{n-m}=\cdots=a_{n+m}>a_{n+m+1}>\cdots>a_{2n},$$
moreover, $m\leq k$.
\end{conj}

We provide some evidence supporting this conjecture in this paper. We first observe that the case of genus two knots can be confirmed easily by using a result of Ozsv\'ath and Szab\'o in \cite{OS03}, or by Jong's inequalities in \cite{Jon10}, which is pointed out to the author by Kunio Murasugi. 
\begin{thm}\label{genus two case}
If $K$ is an alternating knot of genus two, then it satisfies the statement of Conjecture \ref{HMconjecture}.
\end{thm}
\begin{proof}
Note since the trapezoidal conjecture is true for genus two alternating knots, only $m\leq k$ are left to be verified. If $|\sigma(K)|=4=2g(K)$, then Conjecture \ref{HMconjecture} is obviously true since the degree the symmetric Alexander polynomial is less than or equal to $g(K)$. If $|\sigma(K)|=2$, Corollary 1.6 of \cite{OS03} or Theorem 1.6 of \cite{Jon10} implies $a_1\geq 2a_0+1$, hence the conjecture. If $\sigma(K)=0$, Corollary 1.6 of \cite{OS03} or Theorem 1.6 of \cite{Jon10} implies $a_1\geq 2a_2$, and $\Delta_K(1)=1$ implies $a_0=1+2a_1-2a_2$, therefore $a_0>a_1>a_2$.
\end{proof}

Our main result confirms the Hirasawa-Murasugi conjecture for two-bridge knots.

\begin{thm}\label{main}
Conjecture \ref{HMconjecture} is true for two-bridge knots.
\end{thm}

The proof of this theorem is given in Section 3. For the strategy of the proof, we extend Hartley's induction argument in \cite{Har79}. Hartley's induction utilizes extended digrams of two-bridge knots to compute their Alexander polynomials, and for our purpose we further implement Shinohara's algorithm in the induction to compute the signature invariant \cite{Shi76}.\\

\noindent\textbf{Organization.} In Section 2 we recall the preliminaries, which includes computing the Alexander polynomial using extended diagrams and Shinohara's result on the signature of two-bridge knots. Section 3 is devoted for the induction argument: after some further technical preparation for the induction in Subsection 3.1 and Subsection 3.2, the key parts of the proof are carried out in three parallel steps in Subsection 3.3-3.5.\\

\noindent\textbf{Acknowledgment:} I thank Stephan Burton, Matt Hedden, Effie Kalfagianni and Christine Ruey Shan Lee for their interest, and Matt Hedden again for his help on improving the exposition of this work. The author is grateful to Kunio Murasugi for pointing out a mistake in an earlier version of this paper, and suggesting a correction for the proof of Theorem \ref{genus two case}.

\section{Preliminaries}
For the reader's convenience, we recall some preliminaries regarding two-bridge knots (and links) in this section. As it will be clear, all the two-bridge links we consider will come with a preferred orientation, so this allows us to talk about the signature of a two-bridge link without ambiguity. This section has three parts, consisting of the Schubert normal form, extended diagrams and Shinohara's method for computing the signature invariant. In particular, we shall see both the Alexander polynomial and the signature of a two-bridge link can be read off from its extended diagram. 

\textbf{Convention.\ }For the ease of terminology, by the term two-bridge knot we often include the case of links and this shall not cause any confusion. With this convention, Theorem \ref{main} can also be understood as: any two-bridge link with the preferred orientation specified below satisfies Conjecture 1 (see Theorem \ref{reformulation of the main theorem} for a precise reformulation). 

\subsection{Two-bridge knots and their Schubert normal forms}
A two-bridge knot is one that admits a bridge-presentation with two overarcs and two underarcs. Every two-bridge knot can be presented in its \emph{Schubert normal form}. More concretely, given a pair of coprime numbers $(p,q)$ such that $q$ is odd and $2p>q>0$, we may construct a two-bridge knot via following procedure. Firstly we draw two overarcs, placed horizontally on the same level, on which we mark $p+1$ points equidistantly, numbered from $0$ to $p$ with $0$ at the end near the center (see Fig.\ \ref{overarcs}). Then an underarc begins by spiralling out clockwisely from one of the $0$'s, passing under the two overarcs alternatively through the mark points $q$, $2q$, ... When reaching the outside, it makes a turn with a radius within $q/2$, and then spirals counterclockwisely, again passes through the overarcs alternatively under mark points with a $q$-unit difference. This process is repeated until the underarc reaches the tail of some overarc (Fig.\ \ref{(4,3) in Schubert normal form}). The other underarc is drew symmetrically. The so obtained diagram is called the Schubert normal form of the two-bridge knot of type $(p,q)$. Throughout, we orient thus obtained knots (or links) by requiring the orientation of overarcs to be center pointing.

\begin{figure}[htb]
\begin{minipage}[t]{0.5\linewidth}
\centering{
\fontsize{0.5cm}{2em}
\resizebox{65mm}{!}{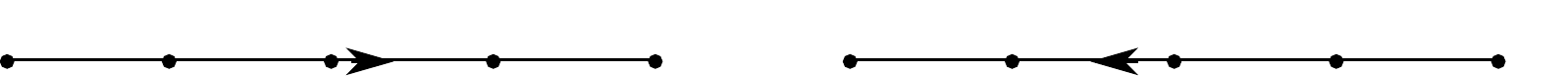}
\caption{$p=4$}
\label{overarcs}
}
\end{minipage}%
\begin{minipage}[t]{0.5\linewidth}
\centering{
\fontsize{0.5cm}{2em}
\resizebox{65mm}{!}{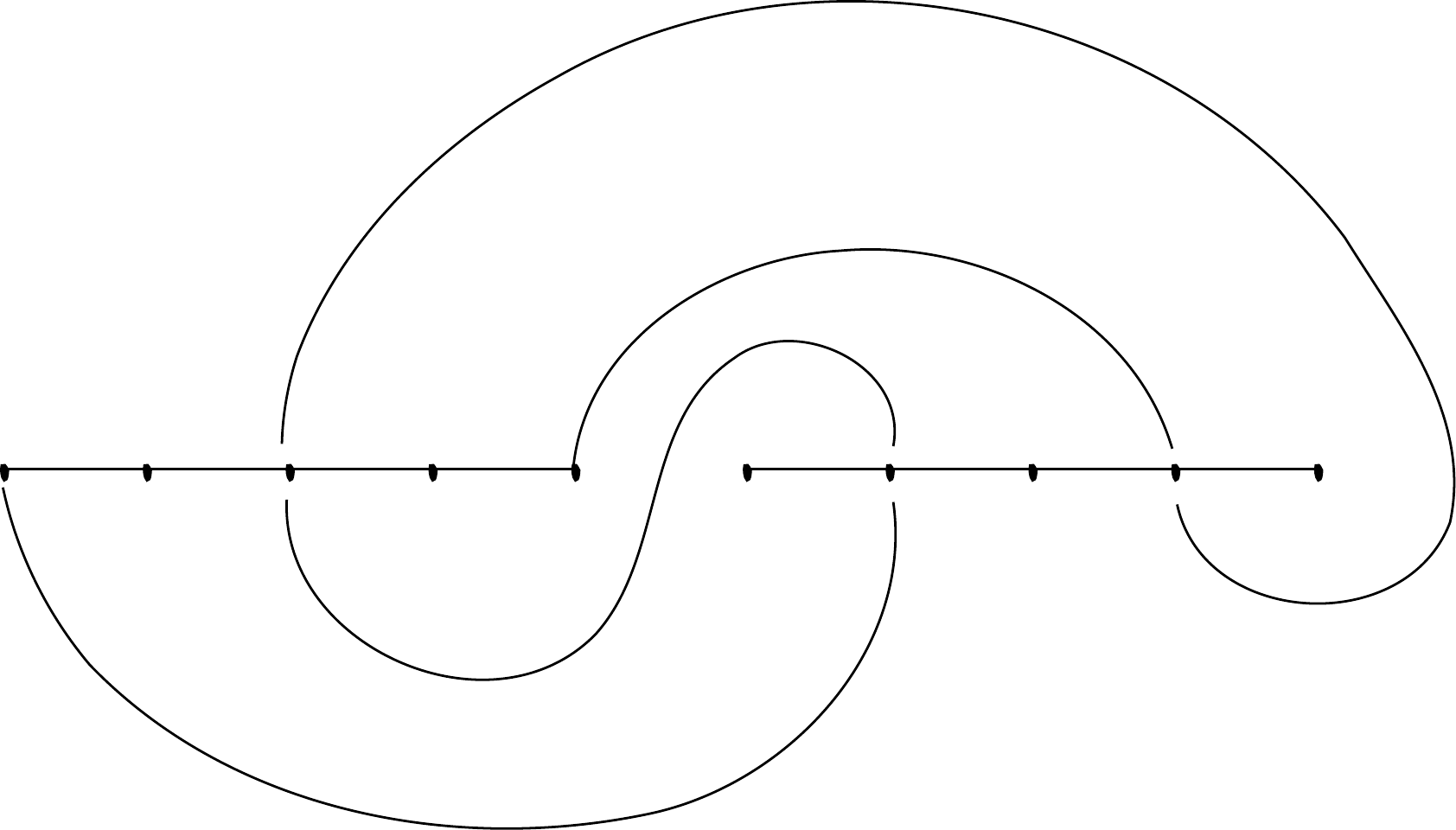}
\caption{$(4,3)$ with one underarc}
\label{(4,3) in Schubert normal form}
}
\end{minipage}%
\end{figure}
\subsection{Extended diagrams and the Alexander polynomial}
Closely related to the Schubert normal form of a two-bridge knot is its \emph{extended diagram}, introduced by Hartley \cite{Har79}. The extended diagram is obtained by unwinding the Schubert normal form: instead of drawing two overarcs horizontally, we draw a number of parallel overarcs, placed vertically, each one is marked off by numbers from $0$ to $p$ from the bottom to the top (see Fig.\ \ref{verticals}). Then starting from $0$ of one overarc, we let the underarc proceed from left to right if it were going clockwisely in the Schubert normal form, and going from right ot left if it were spiralling counterclockwisely (Fig.\ \ref{extended diagram for (4,3)}). 

The main advantage of the extended diagram presentation is that one could read off the (reduced) Alexander polynomial of the corresponding knot directly. To be precise, denote the overarcs which are hit by the underarc by $W_i$, with $i$ goes from $0$ to some number $l$ from left to right, and let $\alpha_i$ be the number of segments joining the $W_i$ and the $W_{i+1}$. By applying Fox calculus to the knot group presentation coming from the Schubert normal form, Hartley proved

\begin{thm}[\cite{Har79}]
 $\Delta(t)=\Sigma_{i=0}^{l-1}(-1)^i\alpha_i t^i$.
\end{thm}
For example, the two bridge link of type $(4,3)$ shown in Fig. \ref{extended diagram for (4,3)} has $\Delta(t)=2-2t$.

More technical results regarding extended diagrams will be needed for the proof of our main theorem, however, we defer that to Section 3 for the ease of reading. 

\begin{figure}[htb]
\begin{minipage}[t]{0.5\linewidth}
\centering{
\fontsize{0.5cm}{2em}
\resizebox{63mm}{!}{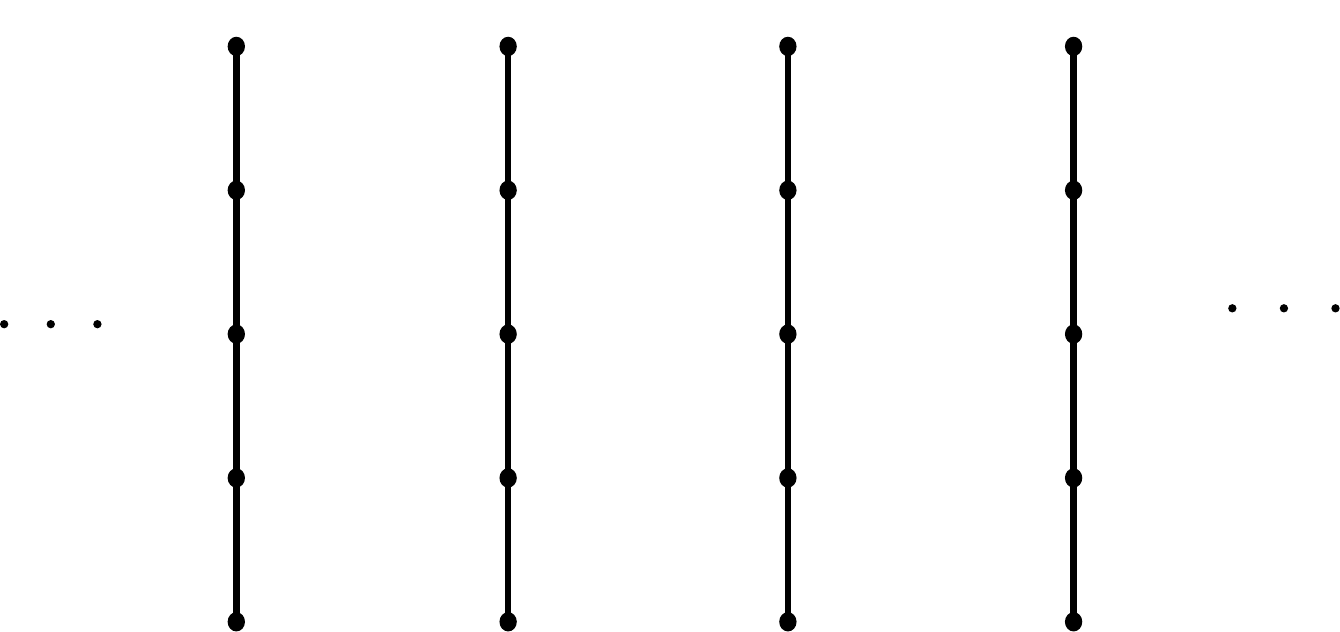}
\caption{overacrs in extended diagram}
\label{verticals}
}
\end{minipage}%
\begin{minipage}[!htb]{0.5\linewidth}
\centering{
\fontsize{0.5cm}{2em}
\resizebox{40mm}{!}{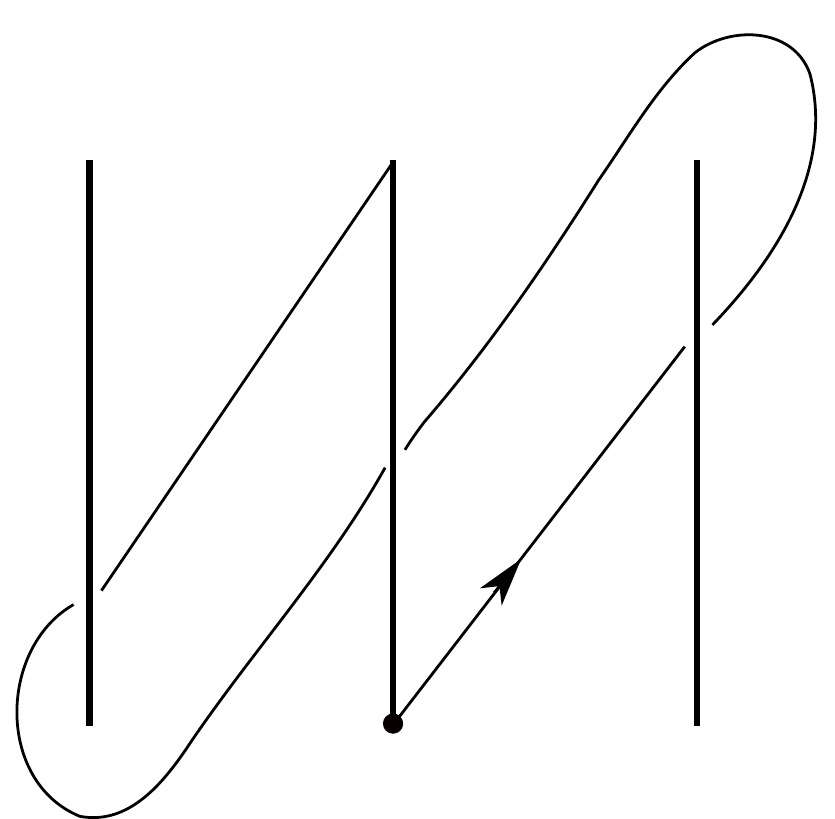}
\caption{Extended diagram for $(4,3)$: $\alpha_0=2$, $\alpha_1=2$, so the $(4,3)$-link has $\Delta(t)=2-2t$.}
\label{extended diagram for (4,3)}
}
\end{minipage}
\end{figure}

\subsection{The signature of two-bridge knots}
Shinohara gave a convenient way of computing the signature invariant of a two-bridge knot from its Schubert normal form \cite{Shi76}. Keep in mind that the knots (especially links) that we consider are oriented. We have
\begin{thm}[\cite{Shi76}]\label{Shinohara}
For a two-bridge knot $K$ of type $(p,q)$, its signature $\sigma(K)$ equals the algebraic sum of the signed crossings of one underarc with the overarcs in its Schubert normal form. 
\end{thm}

In view of the relation between the Schubert normal form and the extended diagram, we may say $\sigma$ equals the algebraic sum of the signed crossings of an underarc with the overarcs in the extended diagram, with the overarcs oriented as downward pointing. For exmaple, the two-bridge knot of type $(4,3)$ has signature $1$ (Fig.\ \ref{(4,3) in Schubert normal form} and Fig.\ \ref{extended diagram for (4,3)}).

\begin{rmk} Denote by $\sigma(p,q)$ the signature of the two-bridge knot of type $(p,q)$, one can deduct the following well-known formula from Theorem \ref{Shinohara}:
$$\sigma(p,q)=\sum_{i=1}^{p-1}(-1)^{[\frac{iq}{p}]}.$$
\end{rmk}
\section{Proof the main theorem}
In this section we give a proof of the main theorem. Overall, the strategy is to carry out an induction on the pairs $(p,q)$, starting from $(1,1)$ using three types of moves $T_i$ that will be defined later, $i=1,2,3$. This is the approach Hartley took to prove the trapezoidal conjecture for two-bridge knots. To facilitate our proof, the first subsection recalls technical results regarding extended diagrams from \cite{Har79}. Then we move to examine the effect of each move $T_i$. Among these, the effect of the $T_1$ move is most subtle and requires a fair amount of technical care, hence the corresponding discussion will be postponed to the last subsection.
\subsection{More technical preparations on extended diagrams}
To carry out the induction, rather than restricting to a pair $(p,q)$ such that $2p>q>0$, Hartley introduced a bigger set consisting of the so-called admissible pairs. 
\begin{defn}
A pair of postive integers $(p,q)$ is said to be admissible if $gcd(p,q)=1$, and $q$ is odd.
\end{defn}
Note given an admissible pair $(p,q)$, we can similarly associate to it an extended diagram. More concretely, first introduce grid lines which consist of infinitely many parallel vertical lines, $W_i$, placed equidistantly, with subindex ranging from $-\infty$ to $\infty$, from left to right. On each grid line, mark $p+q$ points from $-(q-1)/2$ to $p+(q-1)/2$ with higher points having higher index (See Fig.\ \ref{gridlines}). The segments between $0$ and $p$ serve as the overarcs. Secondly, denote the point labeled by $j$ on $W_i$ by $x_{ij}$, and for all $i$, join the following pair of points by pairwise disjoint simple arcs lying within the region bounded by $W_i$ and $W_{i+1}$ (See Fig.\ \ref{arcs between two grid lines}): 
\begin{itemize}
\item $x_{ij}$ and $x_{i+1,j+q}$, where $-(q-1)/2\leq j \leq p-(q+1)/2$
\item  $x_{i+1,j}$ and $x_{i+1,-j}$ (called bottom loops), $x_{i,p-j}$ and $x_{i,p+j}$ (called top loops), where $1\leq j\leq (q-1)/2$
\end{itemize} 
After this, the above simple arcs piece up to give infinitly many underarcs (See Fig.\ \ref{infinitly many underarcs}). Arbitrarily pick a single underarc, by which we call \textit{the principal underarc}. Reindex the overarcs if necessary, so that leftmost overarc hit by the principal underarc is $W_0$. If the rightmost overarc hit by the principal underarc is $W_l$, we call $l$ to be the \textit{length} of $(p,q)$ (See Fig.\ \ref{infinitly many underarcs}).\\

There are two important sequence associated to the extended diagram of $(p,q)$. The first one is the \textit{arc sequence} $\alpha_i$, which is the number of arcs connecting $W_i$ and $W_{i+1}$ and coincides with the coefficients of the Alexander polynomial. The second one is the so-called \textit{bottom sequence} $b_i$, which is equal to twice the number of bottom loops of the principal underarc at $W_i$, plus one if the principal underarc starts at $W_i$ (sometimes we may consider $b_k$ with $k>l$, in this case $b_k$ should be understood as $0$). For example, in Fig.\ \ref{infinitly many underarcs} where the extended diagram of $(4,3)$ is shown, we see $l=2$, $b_0=2$, $b_1=1$ and $b_2=0$.   

\begin{figure}[htb]
\begin{minipage}[t]{0.5\linewidth}
\centering{
\fontsize{1.5cm}{2em}
\resizebox{45mm}{!}{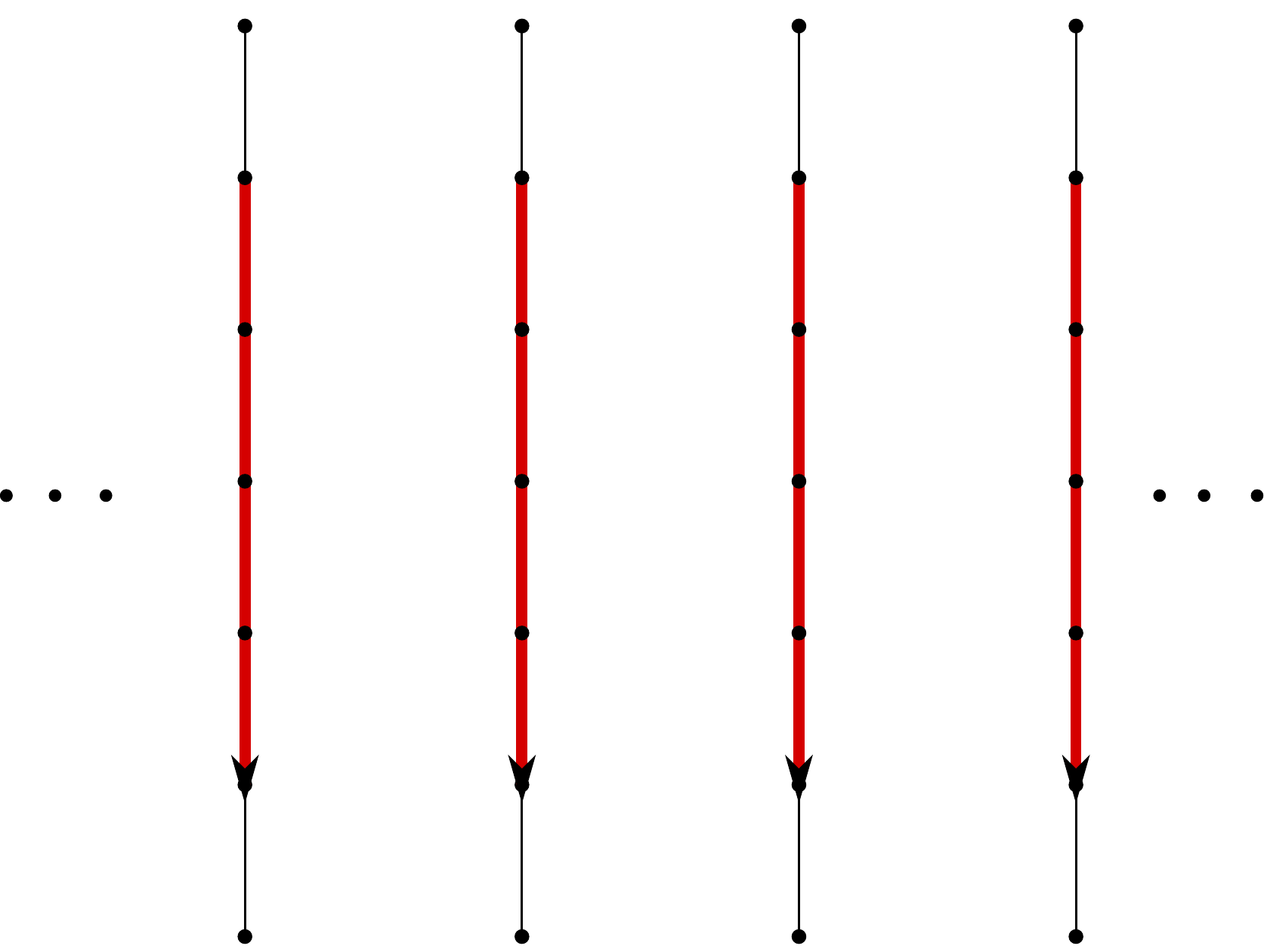}
\caption{grid lines for $(4,3)$}
\label{gridlines}
}
\end{minipage}%
\begin{minipage}[!htb]{0.5\linewidth}
\centering{
\fontsize{0.5cm}{2em}
\resizebox{40mm}{!}{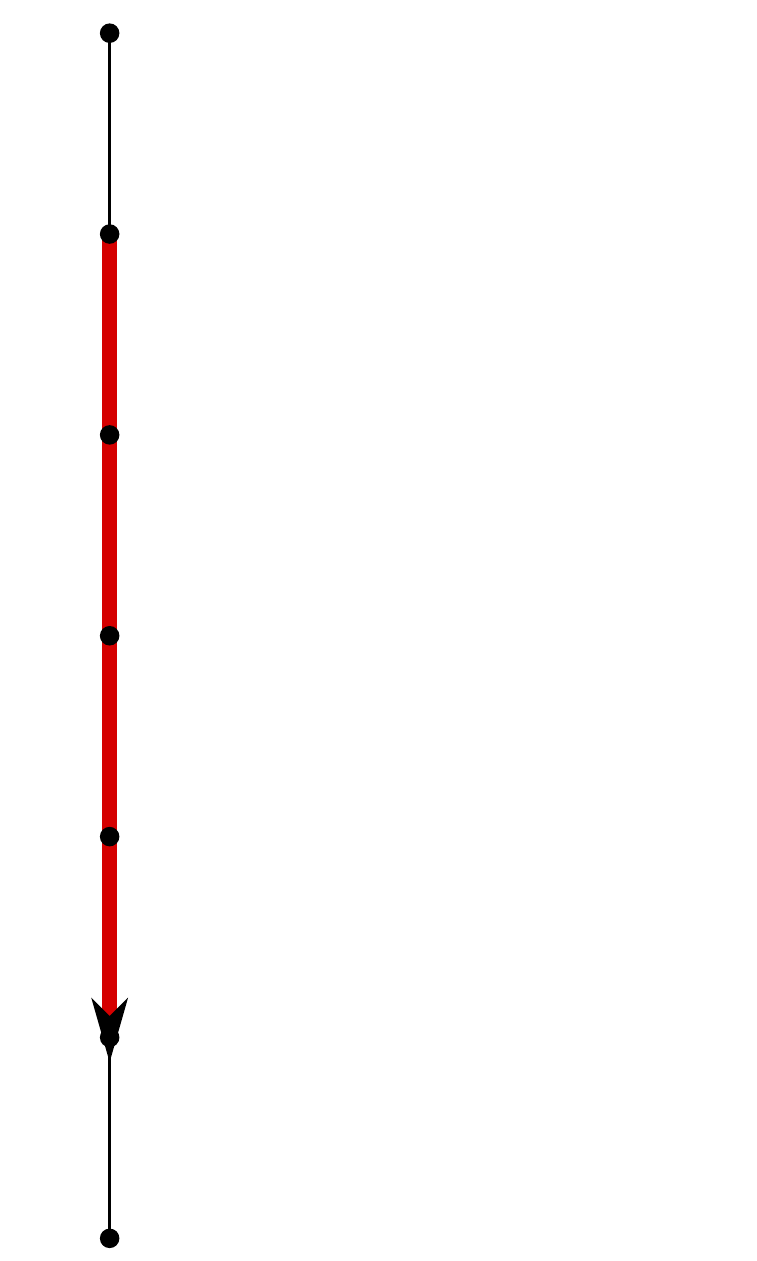}
\caption{arcs between two grid lines}
\label{arcs between two grid lines}
}
\end{minipage}
\end{figure}

\begin{figure}[htb]
\centering{
\fontsize{1cm}{2em}
\resizebox{85mm}{!}{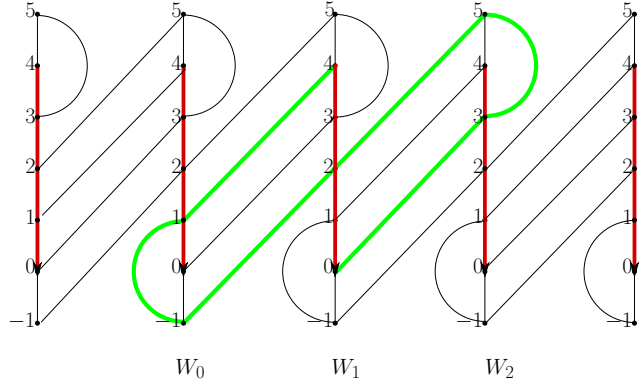}
\caption{Revisiting the extended diagram for $(4,3)$: the thickened (and green) underarc is the principle underarc; the length of  $(4,3)$ is $2$.}
\label{infinitly many underarcs}
}
\end{figure}

Key technical results regarding $\alpha_i$ and $b_i$ are summarized below.
\begin{prop}[\cite{Har79}]
Let $(p,q)$, $\alpha_i$ and $b_i$ be as above. Then  $$\alpha_i-\alpha_{i-1}=b_i-b_{l-i},\ 1\leq i \leq l.$$ Moreover, $\{b_i\}$ satisfies the following three so-called IH properties:
\begin{itemize}
\item[(IH1)] There is an interger $h$ satisfying $1\leq h \leq l$ and an integer $r\leq h$ such that $b_i=0$ when $i>h$, and $0\leq S_0< S_1<\cdots<S_r=S_{r+1}=\cdots=S_h$, where $S_{2j}=b_{h-j}$ and $S_{2j+1}=b_j$. 
\item[(IH2)] If $h^*\geq h$ and $2j\leq h^*$, then $b_j\geq b_{h^*-j}$.
\item[(IH3)] If $0\leq i<j$ and $b_i=b_j$, then $b_i=b_k=b_j$ for all $k$ such that $i\leq k \leq j$.
\end{itemize}
\end{prop}

Now let us introduce the $T_i$ moves that we promised in the beginning of this section. These are operations on the admissible pairs defined as: 
\begin{align*}
&T_1:(p,q)\longmapsto (p+q,q)\\
&T_2:(p,q)\longmapsto (p,2p+q)\\
&T_3:(p,q)\longmapsto (p,2p-q),
\end{align*}
where $T_3$ is only defined when $p>q$, hence $T_3$ cannot be applied after $T_2$ or $T_3$, but only after $T_1$. One nice feature of the $T_i$ moves is the following.
\begin{prop}[\cite{Har79}]\label{inductible lemma}
Any admissible pair $(p,q)$ can be obtained from $(1,1)$ via applying a sequence of $T_i$, $i=1,2,3$.
\end{prop} 

In fact, the IH properties are proved inductively using these $T_i$ moves, and they imply the trapezoidal conjecture for two-bridge knots. For our purpose, the following facts will be important.
\begin{prop}[\cite{Har79}]\label{change of bottom sequence}
Given an admissible pair $(p,q)$, let $l$, $b_i$ and $\alpha_i$ denote the length, bottom sequence and the arc sequence respectively. Then the length and bottom sequence of $T_i(p,q)$, $i=1,2,3$ are summarized in the following table 
\begin{center}
\begin{tabular}{|c|c|c|}
\hline
 & length $l'$ & bottom sequence $b_i'$  \\
\hline
$T_1(p,q)$ & $l+1$ & $b_i$ \\
\hline
$T_2(p,q)$ & $l$ & $2\alpha_i+b_i$ \\
\hline
$T_3(p,q)$ & $l$ & $2\alpha_i-b_i$ \\
\hline
\end{tabular}
\end{center}
\end{prop}

\subsection{Reformulation of the main theorem}
Notice that an admissible pair $(p,q)$ may give rise to a two-component link. Therefore during this process of induction, both the degree of the Alexander polynomial and the signature may change their parity, so we would like to adjust the statement of the Hirasawa-Murasugi conjecture to take care of this issue.
\begin{thm}\label{reformulation of the main theorem}
Let $(p,q)$ be an admissible pair, $\sigma$ be its signature, and $\Delta_K(t)=a_0-a_1t+\cdots+(-1)^{l-1}a_{l-1}t^{l-1}$ be its Alexander polynomial, where $a_i>0$, $i=0,...,l-1$. Then
\begin{equation}\label{Fox inequality}
a_0<a_1<\cdots<a_{i_0-1}=a_{i_0}=\cdots=a_{l-i_0}>a_{l-i_0+1}>\cdots>a_{l-1}.
\end{equation}
Moreover, 
\begin{equation}\label{HM inequality}
\lfloor \frac{|\sigma|+1}{2}\rfloor \geq \lfloor \frac{l-2(i_0-1)}{2}\rfloor.
\end{equation}
\end{thm}

Here $\lfloor \cdot \rfloor$ is understood as taking the maximal integer part. It is obvious that the above theorem implies Theorem \ref{main}.
\begin{proof}
Note Inequality (\ref{Fox inequality}) is already proved by Hartley, so in the rest of this section, we will focus on proving Inequality (\ref{HM inequality}). To do that, we begin by noticing it is obviously true for the pair $(1,1)$. In view of Prop.\ \ref{inductible lemma}, we just need to see that if a pair $(p,q)$ satisfies Theorem \ref{reformulation of the main theorem}, so is $T_i(p,q)$ for $i=1,2,3$. This is done in subsections 3.3-3.5.
\end{proof}

\begin{rmk}
From now on, we call $\lfloor \frac{l-2(i_0-1)}{2}\rfloor$ the radius of the stable terms of the Alexander polynomial. Note that the case in which all coefficients are distinct could happen, and in that case, $i_0-1=l-i_0$, which implies the radius is zero.
\end{rmk}

\subsection{The effect of $T_2$ move}
This subsection is devoted to proving the following statement.
\begin{prop}
If Theorem \ref{reformulation of the main theorem} is true for an admissible pair $(p,q)$, then it is true for $T_2(p,q)$.
\end{prop}
\begin{proof}
Let $\alpha_i$, $b_i$, $l$ denote the number of connecting arcs, bottom sequence, and length for $(p,q)$, and let $\alpha_i'$, $b_i'$, $l'$ be the corresponding quantities for $T_2(p,q)=(p,2p+q)$. By Prop.\ \ref{change of bottom sequence} $b_i'=2\alpha_i+b_{l-i}$ and $l'=l$, hence we have
$$\begin{aligned}
\alpha_i'-\alpha_{i-1}' &= b_i'-b_{l'-i}'\\
&=(2\alpha_i+b_{l-i})-(2\alpha_{l-i}+b_{i})\\
&=2(\alpha_i-\alpha_{l-i})-(b_i-b_{l-i})\\
&=2(\alpha_i-\alpha_{i-1})-(b_i-b_{l-i})\\
&=2(\alpha_i-\alpha_{i-1})-(\alpha_i-\alpha_{i-1})\\
&=\alpha_i-\alpha_{i-1},
\end{aligned}$$
where in the $4th$ equality we used $\alpha_{l-i}=\alpha_{i-1}$ due to the symmetry of the Alexander polynomial. So the radius $m=\lfloor \frac{l-2(i_0-1)}{2}\rfloor$ does not change after the $T_2$ move.

The signature invariant is also unchanged after the $T_2$ move. To see this, note
$$\sigma(p,q)=\sum_{i=1}^{p-1}(-1)^{\lfloor \frac{iq}{p}\rfloor }=\sum_{i=1}^{p-1}(-1)^{\lfloor \frac{i(q+2p)}{p} \rfloor}=\sigma(p,2p+q).$$

Therefore, Theorem \ref{reformulation of the main theorem} is true for $T_2(p,q)$ provided it is true for $(p,q)$. 
\end{proof}

\subsection{The effect of $T_3$ move}
In this subsection we examine the effect of $T_3$. 
\begin{prop}
If Theorem \ref{reformulation of the main theorem} is true for an admissible pair $(p,q)$, then it is true for $T_3(p,q)$.
\end{prop}

\begin{proof}
Let $\alpha_i$, $b_i$, $l$ denote the number of connecting arcs, bottom sequence, and length for $(p,q)$, and let $\alpha_i'$, $b_i'$, $l'$ be the corresponding quantities for $T_3(p,q)=(p,2p-q)$. In this case, we have $b_i'=2\alpha_i-b_i$ and $l'=l$ by Prop.\ \ref{change of bottom sequence}. Therefore,
$$\begin{aligned}
\alpha_i'-\alpha_{i-1}' &= b_i'-b_{l'-i}'\\
&=(2\alpha_i-b_{i})-(2\alpha_{l-i}-b_{l-i})\\
&=2(\alpha_i-\alpha_{l-i})-(b_i-b_{l-i})\\
&=2(\alpha_i-\alpha_{i-1})-(b_i-b_{l-i})\\
&=2(\alpha_i-\alpha_{i-1})-(\alpha_i-\alpha_{i-1})\\
&=\alpha_i-\alpha_{i-1}\\
\end{aligned}$$

So the radius $m=\lfloor \frac{l-2(i_0-1)}{2}\rfloor$ does not change after the $T_3$ move.

For the signature, we have
$$\sigma(p,2p-q)=\Sigma_{i=1}^{p-1}(-1)^{\lfloor\frac{i(2p-q)}{p}\rfloor}=\Sigma_{i=1}^{p-1}(-1)^{\lfloor\frac{i(-q)}{p}\rfloor}=-\Sigma_{i=1}^{p-1}(-1)^{\lfloor\frac{iq}{p}\rfloor}=-\sigma(p,q).$$
Therefore, neither does $|\sigma|$ change after the $T_3$ move. Hence Theorem \ref{reformulation of the main theorem} is true for $T_3(p,q)$ provided it is true for $(p,q)$.
\end{proof}

\subsection{The effect of $T_1$ move}
In this subsection we will discuss the effect of $T_1$. Note on the level of knots, $T_2$ preserves the knot, and $T_3$ changes the knot to its mirror, and that is the reason these two cases are relatively easier compared to case of $T_1$. The goal of this subsection is to prove
\begin{prop}\label{T_1 is good}
If Theorem \ref{reformulation of the main theorem} is true for an admissible pair $(p,q)$, then it is true for $T_1(p,q)$.
\end{prop}
The proof of this proposition will come at the end of this subsection, after investigating the effect of $T_1$ on the signature and the Alexander polynomial.

First of all, we present the effect of $T_1$ on the signature.
\begin{lem}\label{effect of T_1 on signature}
$\sigma(T_1(p,q))-\sigma(p,q)=1$.
\end{lem}

\begin{figure}
\begin{minipage}[t]{0.5\linewidth}
\centering{
\fontsize{1cm}{2em}
\resizebox{40mm}{!}{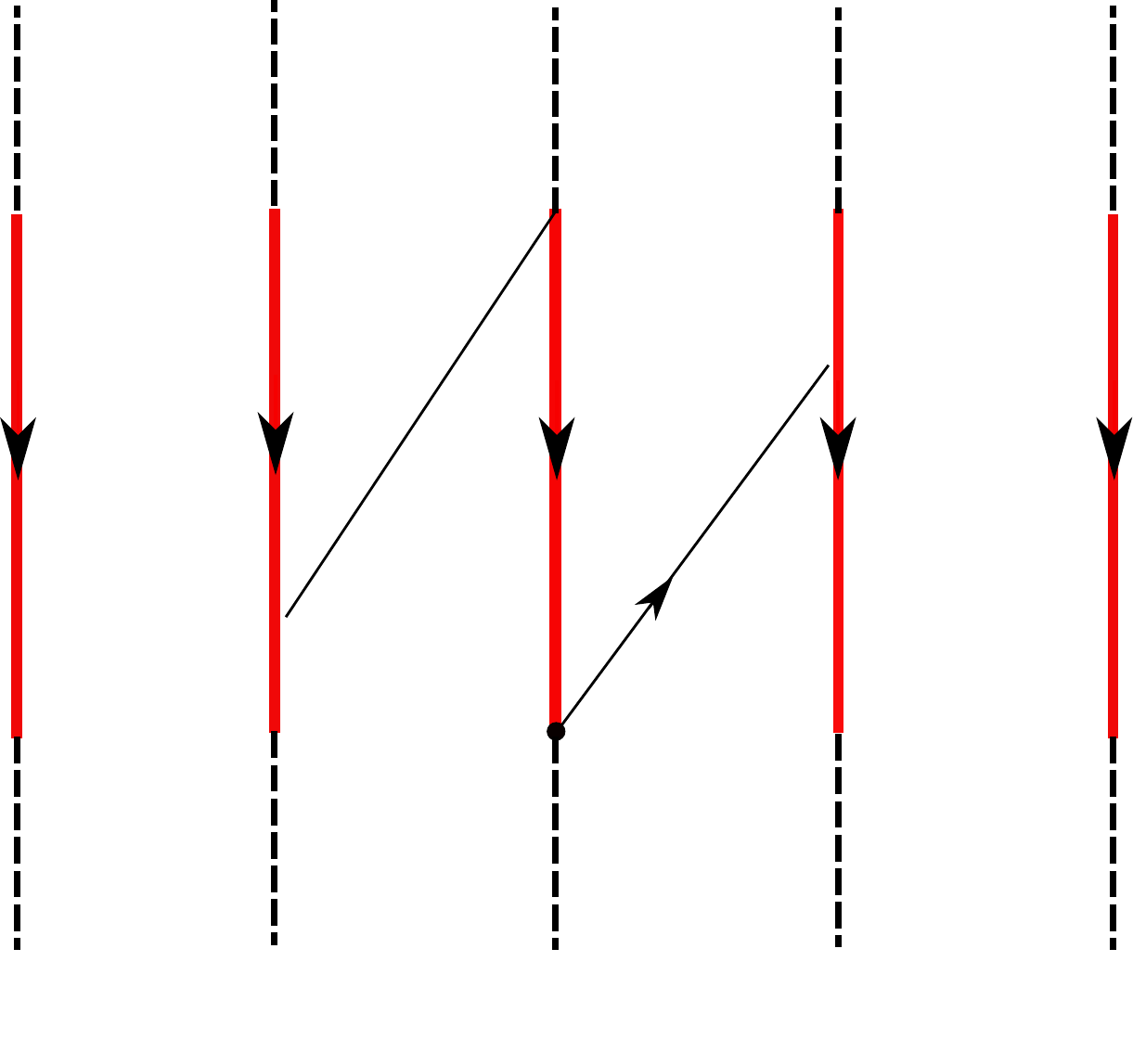}
\caption{Before $T_1$}
\label{BeforeT1}
}
\end{minipage}%
\begin{minipage}[t]{0.5\linewidth}
\centering{
\fontsize{1cm}{2em}
\resizebox{50mm}{!}{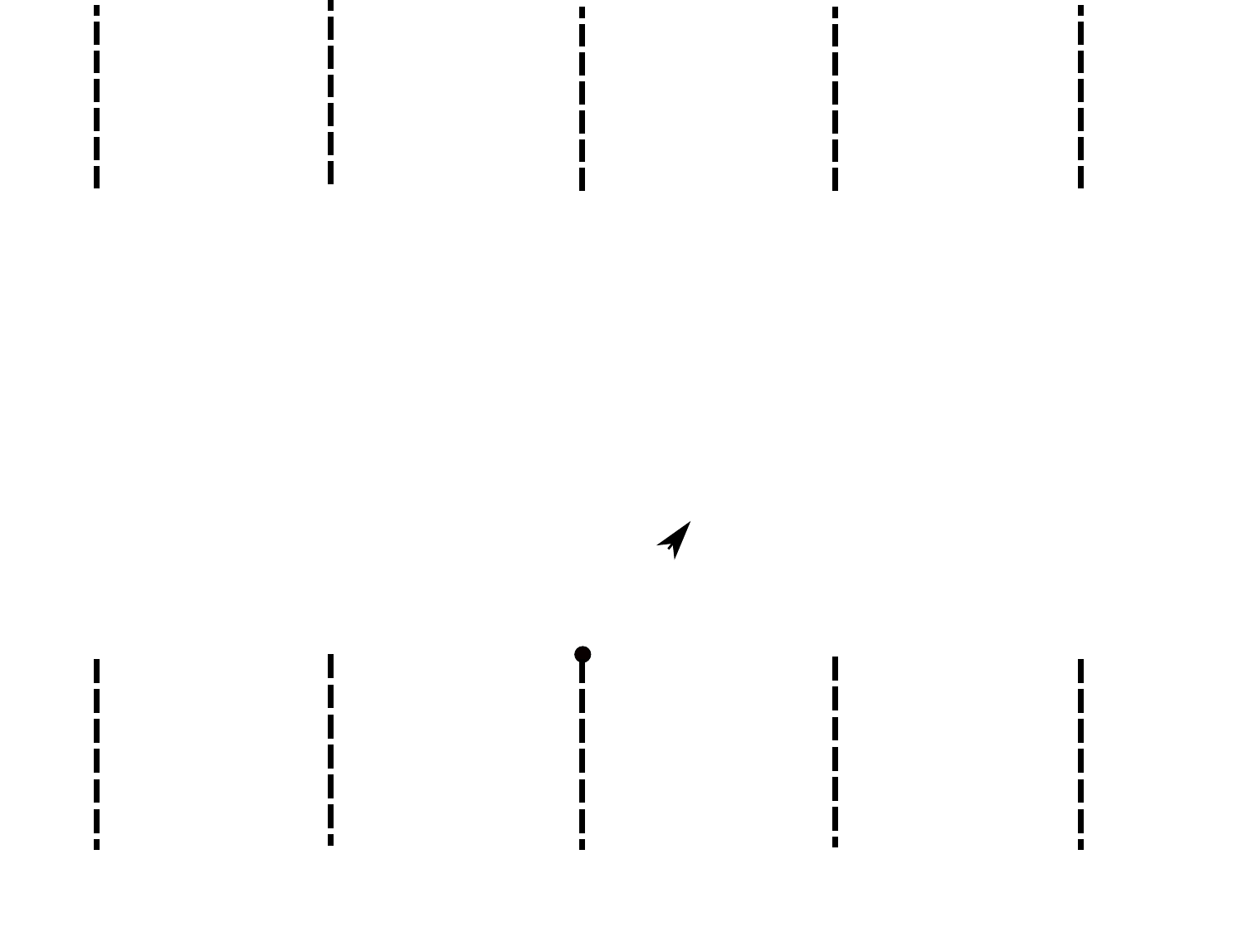}
\caption{After $T_1$}
\label{AfterT1}
}
\end{minipage}
\end{figure}
\begin{figure}
\centering{
\fontsize{1cm}{2em}
\resizebox{85mm}{!}{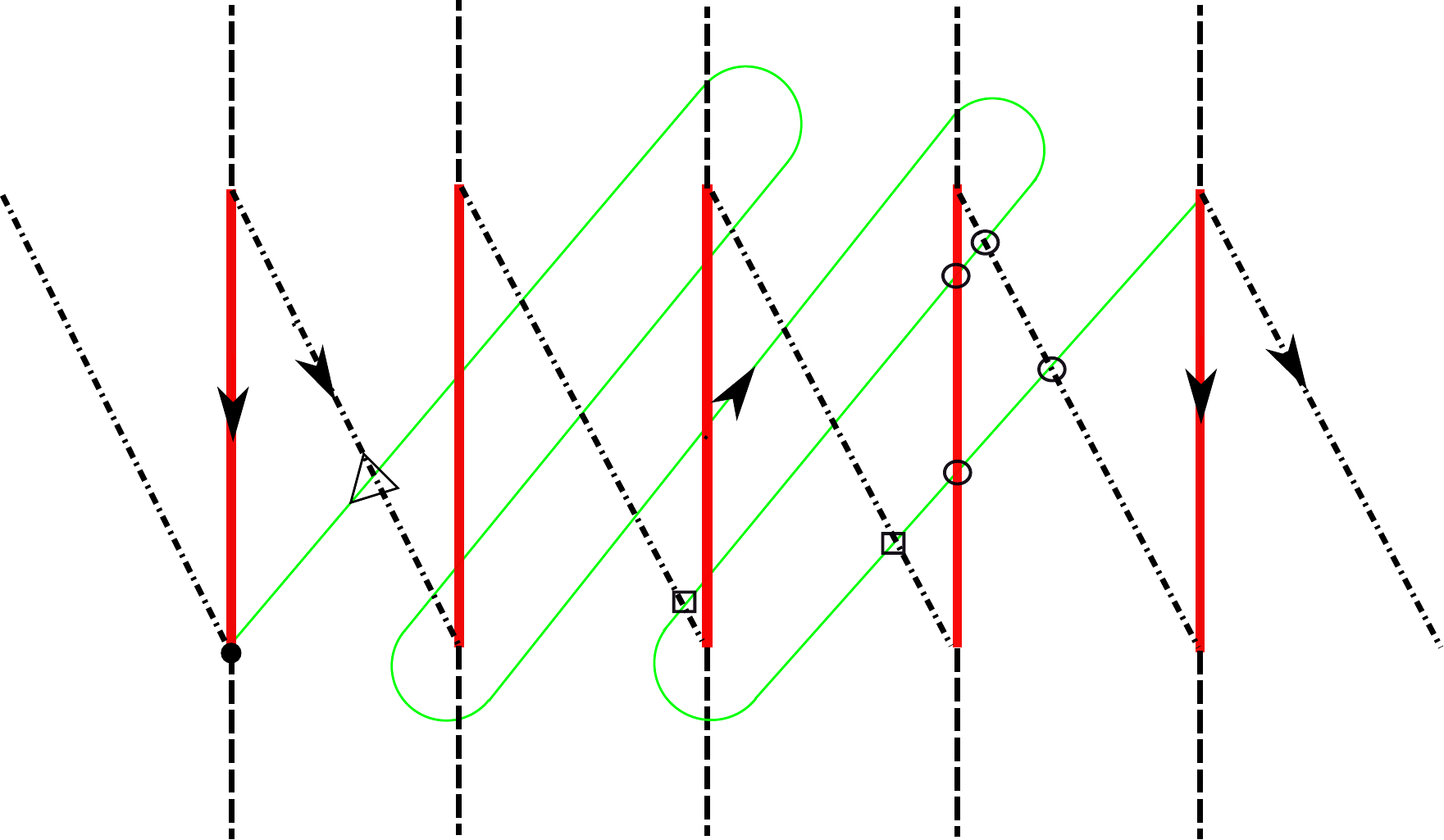}
\caption{The effect of $T_1$ on signature}
\label{signatureT1}
}
\end{figure}
\begin{proof}
The effect of $T_1$ on the extended diagram is shown in Fig.\ \ref{BeforeT1} and Fig.\ \ref{AfterT1}. We describe the effect as sliding the bottom end of the parallel overarcs one unit to the right.

First compare the new overarcs and the old ones pair by pair who share the same top, starting from right to left. We see that the crossings between an old overarc and the principle underarc has counterparts in the crossings between the new overarc and the principle underarc (the crossing which are circled in Fig.\ \ref{signatureT1}). Secondly, the presence of each bottom circle in the principle underarc creates two new crosing with the new overarcs (the crossings which are boxed in Fig.\ \ref{signatureT1}), yet these two crossings cancel each other algebraically. Finally, sliding the overarc on which the principle underarc starts creates a new crossing (the crossing marked by a triangle in Fig.\ \ref{signatureT1}), and this crossing has positive sign. The conclusion then follows from comparing the sum of the signed crossings in view of Theorem \ref{Shinohara}. 
\end{proof}

Next we examine how the radius of the stable terms behave under the $T_1$ operation. We separate the discussion into two cases. First we have
\begin{prop}\label{radius change if there were no stable term}
Given an admissible pair $(p,q)$, if there are no stable terms in the corresponding Alexander polynomial, then there are exactly two stable terms in the Alexander polynomial corresponding to $T_1(p,q)$.
\end{prop}
\begin{proof}
In this case, in view of symmetry of the coefficients, the degree of the Alexander polynomial must be even, and hence $l=l(p,q)$ is odd. Let $l=2k+1$ and the coefficients of the Alexander polynomial for $(p,q)$ be $\alpha_0$,...,$\alpha_{k-1}$, $\alpha_{k}$,$\alpha_{k+1}$,...,$\alpha_{2k}$. After the $T_1$ move, denote the coefficients by $\alpha_0'$,...,$\alpha_{k}'$,$\alpha_{k+1}'$,...,$\alpha_{2k+1}'$. We have
\begin{displaymath}
\alpha_{k+1}'-\alpha_{k}'=b_{k+1}'-b_{(2k+2)-(k+1)}'=b_{k+1}-b_{k+1}=0,
\end{displaymath}
and
\begin{displaymath}
\alpha_{k}'-\alpha_{k-1}'=b_{k}'-b_{(2k+2)-(k)}'=b_{k}-b_{k+2}\geq b_{k}-b_{k+1}=\alpha_{k}-\alpha_{k-1}>0
\end{displaymath}
where we used $b_{k+1}\geq b_{k+2}$ due to the second IH property. So the only stable terms are $\alpha_{k+1}'$ and $\alpha_{k}'$.
\end{proof}
 
Secondly, when there were stable terms in the Alexander polynomial before we apply $T_1$, we have the following proposition.
 \begin{prop}\label{radius change when there are stable terms}
Let $\alpha_i$ and $\alpha_i'$ be the coefficients of the Alexander polynomial corresponding to $(p,q)$ and $T_1(p,q)$ respectively. Denote by $l$ the length of $(p,q)$. Let $i_0$, $i_0'$ be integers such that
$$\alpha_0<\alpha_1<\cdots<\alpha_{i_0-1}=\alpha_{i_0}=\cdots=\alpha_{l-i_0}>\alpha_{l-i_0+1}>\cdots>\alpha_{l-1},$$
and
$$\alpha_0'<\alpha_1'<\cdots<\alpha_{i_0'-1}'=\alpha_{i_0'}'=\cdots=\alpha_{l+1-i_0'}'>\alpha_{l-i_0'+2}'>\cdots>\alpha_{l}'.$$
If $i_0-1<l-i_0$, then one of the following statement is true:
\begin{enumerate}
\item  $i_0'=i_0+1$\\
\item  $i_0'=i_0$ and $b_{i_0}=b_{i_0+1}=\cdots =b_l=0$.
\end{enumerate}
 \end{prop}
 
 \begin{proof}
Note that $\alpha_{i_0}-\alpha_{i_0-1}=b_{i_0}-b_{l-i_0}=0$ and hence $b_{i_0}=b_{i_0+1}=\cdots=b_{l-i_0}$ by the third IH property. If $l-i_0>i_0$, then $\alpha_{i_0+1}'-\alpha_{i_0}'=b'_{i_0+1}-b'_{l'-(i_0+1)}=b_{i_0+1}-b_{l-i_0}=0$. Therefore, $i_0'\leq i_0+1$. If $l-i_0=i_0$, then $i_0'\leq i_0+1$ by considering the degree and symmetry of the Alexander polynomial. We move to see $i_0'\geq i_0$.

If $i_0'< i_0+1$, we have $\alpha_{i_0}'-\alpha_{i_0-1}'=b_{i_0}-b_{l-i_0+1}=0$, then by third IH property, $b_{i_0}=b_{i_0+1}=\cdots=b_{l-i_0+1}$. We continue the discussion in two cases. 

\textbf{(Case 1) }If there is $\alpha_{i_0-2}$ term, i.e.\ $i_0\geq 2$, then since $b_{i_0-1}-b_{l-i_0+1}=\alpha_{i_0-1}-\alpha_{i_0-2}>0$, we learn that $b_{i_0-1}>b_{i_0}=b_{i_0+1}=\cdots=b_{l-i_0+1}$. Then $\alpha_{i_0-1}'-\alpha_{i_0-2}'=b_{i_0-1}-b_{l-i_0+2}\geq b_{i_0-1}-b_{l-i_0+1}>0$. Here $b_{l-i_0+1}\geq b_{l-i_0+2}$ follows from the second IH property. So in this case $i_0'=i_0$. Now let $h$ be the integer in the first IH property. If $h<l-i_0+1$, then by the first IH property, $b_{l-i_0+1}=\cdots=b_l=0$ and therefore $b_{i_0}=b_{i_0+1}=\cdots=b_l=0$. We claim $h$ cannot be greater than or equal to $l-i_0+1$. If not, $h\geq l-i_0+1$, then for some $j$ we have $S_j=b_{l-i_0+1}$, then $S_{j+1}$ must be $b_k$ for some $k\leq i_0-1$; otherwise, we have $b_{i_0-1}\leq b_{l-i_0+1}$ in view of first IH property, but this contradicts our earlier observation that $b_{i_0-1}>b_{l-i_0+1}$. This understood, we further oberserve that since $S_{j+2}=b_{l-i_0}=b_{l-i_0+1}=S_j$ (the existence of $S_{j+2}$ follows from the assumption $i_0-1<l-i_0$), so we have $S_j=S_{j+1}$ by the first IH property. Therefore $b_k=b_{l-i_0+1}$, and since $k\leq i_0-1\leq l-i_0+1$ we have $b_{i_0-1}=b_{l-i_0+1}$ by the third IH property. However, this contradicts $b_{i_0-1}>b_{l-i_0+1}$, so $h$ cannot be greater than or equal to $l-i_0+1$. In summary, we have $i_0'=i_0$ and $b_{i_0}=b_{i_0+1}=\cdots=b_{l}=0$.

\textbf{(Case 2) }If there is no $\alpha_{i_0-2}$ term, i.e.\ $i_0=1$, then $i_0'=1$ and $\alpha_0=\alpha_1=\cdots=\alpha_{l-1}$. Since $0=\alpha_1-\alpha_0=b_1-b_{l-1}$, by the third IH property, we have $b_1=b_2=\cdots=b_{l-1}$. Moreover, $\alpha_1'-\alpha_0'=b_1-b_l=0$ implies $b_1=\cdots=b_l$. If $b_l$ is not zero, then the first IH property implies $b_l\leq b_0\leq b_{l-1}=b_l$, and hence all the $b_i's$ are equal by the third IH property, but this is impossible since by definition one and only one of the $b_i$'s is odd and all the others are even. Therefore, we must have $b_1=\cdots=b_l=0$, with $b_0$ being the only nonzero term.

In summary, after $T_1$, the starting index of the stable terms either stays or increases by 1, and moreover, when the starting index stays, more than half of the $b_i$ sequence are zero.
\end{proof}

To prove Prop.\ \ref{T_1 is good}, we need further control of the signature in the case when the starting index of the stable terms stays level. This is what the following proposition addresses.

\begin{prop}\label{signature is positive when half of b_i vanish}
Let $(p,q)$ be an admissible pair with $\{b_i\}$ such that
\begin{displaymath}
b_0, b_1,..., b_{i_0-1}>0,
\end{displaymath}
\begin{displaymath}
b_{i_0}=\cdots =b_l=0
\end{displaymath}
where $l=l(p,q)$ and $i_0\leq \lfloor \frac{l}{2}\rfloor$. Then $\sigma(p,q)\geq 0$.
\end{prop}

To prove the proposition, we need four lemmas.
\begin{lem}\label{signature is less than the bredth of the alexander poly}
For any admissible pair $(p,q)$, let $\sigma=\sigma(p,q)$ and $l=l(p,q)$, then $|\sigma|\leq l-1$.
\end{lem}
\begin{proof}
Note that this statement is true for $(1,1)$. Note $T_2$ and $T_3$ does not change $|\sigma|$ or $l$, while $T_1$ increase $l$ by one and increase $|\sigma|$ at most by one. So inductively, we can show the statement is true for all admissible pairs in view of Prop.\ \ref{inductible lemma}. 
\end{proof}

\begin{lem}\label{$T_2(p,q)$ has no zero terms in its bottom sequence}
$T_2(p,q)$ has no zero terms in its bottom sequence.
\end{lem}
\begin{proof}
Let ${b_i}$ stands for the bottom sequence for $(p,q)$, and ${b_i'}$ for $T_2(p,q)$, $l=l(p,q)=l'(T_2(p,q))=l'$. Then $b_i'=2\alpha_i+b_{l-i}\geq 2\alpha_i>0$ when $i\leq l-1$ and $b_l'=2\alpha_l+b_0=b_0>0$.
\end{proof}

\begin{lem}\label{effect of T3T1}
Let $\{b_i'\}_{0 \leq i\leq l+1}$ be the bottom sequence for $T_3\circ T_1(p,q)$. Then the only zero term is $b_{l+1}'$.
\end{lem}

\begin{figure}
\centering{
\fontsize{0.5cm}{2em}
\resizebox{10cm}{!}{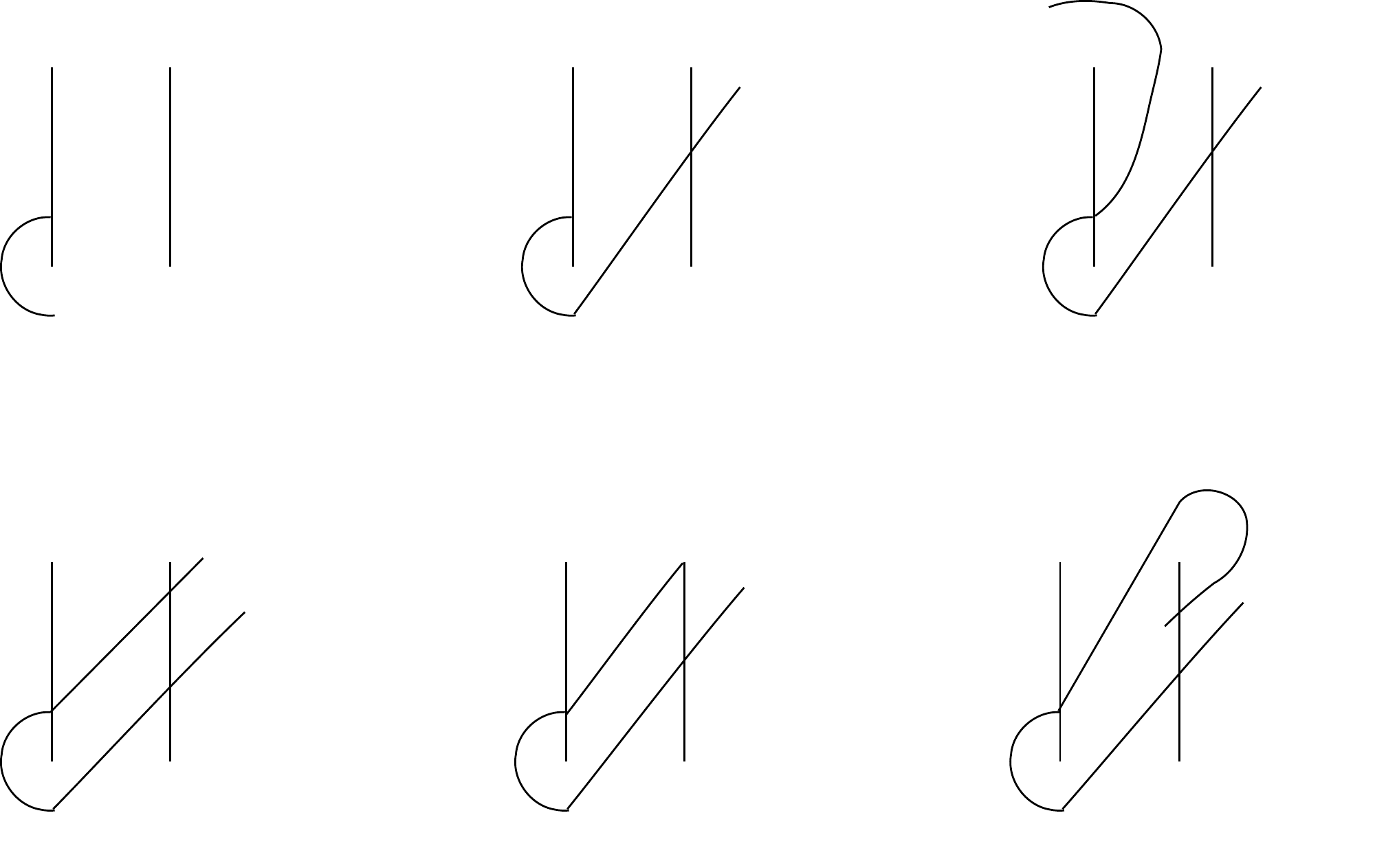}
\caption{}
\label{T3T1}
}
\end{figure}

\begin{proof}
Let $l=l(p,q)$, and for $T_1(p,q)$, we denote by 
\begin{displaymath}
b_0, b_1, ..., b_l, 0
\end{displaymath}
\begin{displaymath}
\alpha_0, \alpha_1,...,\alpha_{l}
\end{displaymath}
the bottom sequence and connecting arc sequence.

So $b'_{l+1}=2\alpha_{l+1}-b_{l+1}=0-0=0$. We claim that $b'_l=2\alpha_l-b_l>0$, hence by the first IH property we know that $b_{l+1}'$ is the only zero term. To see $b'_l>0$, note $T_1(p,q)=(p+q,q)=(p',q')$, so $p'>q'>0$. First, if $b_l$ is zero or odd, then we are done since $2\alpha_l$ will never be zero or odd. Second, if $b_l$ is positive and even, see Fig.\ref{T3T1}: the lower arc joining a bottom loop of $(p',q')$ at $W_l$ must hit $W_{l+1}$ since $p'>q'$ (see Fig.\ \ref{T3T1}.(b)); the upper arc joining the bottom loop cannot turn over $W_l$ (See Fig.\ \ref{T3T1}.(c)), for that would imply $p'\leq \lfloor q'/2\rfloor+\lfloor q'/2\rfloor\leq q'$; therefore, what left are three possiblities (see the second row of Fig.\ \ref{T3T1}) and in all three cases, the existence of one bottom circle gives rise to at least two connecting arcs between $W_l$ and $W_{l+1}$, hence $\alpha_l\geq b_l$, which implies $b'_l>0$. 
\end{proof}

\begin{lem}
Starting from $(1,1)$, to obtain an admissible pair $(p,q)$ with ${b_i}$ such that
\begin{displaymath}
b_0, b_1,..., b_{i_0-1}>0,
\end{displaymath}
\begin{displaymath}
b_{i_0}=\cdots =b_l=0
\end{displaymath}
and $i_0\leq \lfloor \frac{l}{2} \rfloor$,
then there must be at least $(l-i_0)$ $T_1$'s successively in the end, i.e.\ $(p,q)=T_1^{l-i_0}(p',q')$ for some $(p',q')$.
\end{lem}

\begin{proof} 
Note that there are $l-i_0+1$ zero terms in the bottom sequence of $(p,q)$. The operator cannot end with $T_2$ in view of Lemma \ref{$T_2(p,q)$ has no zero terms in its bottom sequence}. If the operator end with $T_3$, then since by definition $T_3$ cannot be applied successively or after $T_2$, nor can it be applied to $(1,1)$ directly, there must be a $T_1$ before it and hence $l\geq 2$. In view of Lemma \ref{effect of T3T1}, the pair we get will have only one zero term in its bottom sequence, not satisfying the assumption that more than half of the $b_i$'s are zero. So $(p,q)=T_1^k(p',q')$ for some $k$. If there is a $T_2$ before $T_1^k$, then $k=(l-i_0+1)$. If there is $T_3$ before $T_1^k$ or $(p',q')=(1,1)$, then there is already a zero in the bottom sequence and hence $k=l-i_0$.
\end{proof}

Now we are ready to give a proof to Prop.\ \ref{signature is positive when half of b_i vanish}.

\begin{proof}[Proof of Prop.\ \ref{signature is positive when half of b_i vanish}]
By the lemma above, $(p,q)=T_1^{l-i_0}(p',q')$. Let $l=l(p,q)$, $l'=l'(p',q')$. Then $l'=l-(l-i_0)=i_0$ and $|\sigma'(p',q')|<i_0$ in view of Lemma \ref{signature is less than the bredth of the alexander poly}. So by Prop.\ \ref{change of bottom sequence} we have $\sigma(p,q)=\sigma'(p',q')+(l-i_0)\geq l-2i_0\geq 0$.
\end{proof}

Finally, with all these preparations we are ready to prove Prop.\ \ref{T_1 is good}, hence concluding the proof of the main theorem.
\begin{proof}[Proof of Prop.\ \ref{T_1 is good}]
Recall we want to prove Theorem \ref{reformulation of the main theorem} is true for $T_1(p,q)$ provided it is true for $(p,q)$. Let $l=l(p,q)$, $l'=l(T_1(p,q))=l+1$. We separate the discussion into two cases.

\textbf{(Case 1) }If there are no stable terms in the Alexander polynomial corresponding to $(p,q)$, then by Prop.\ \ref{radius change if there were no stable term}, there are exactly two stable terms in the Alexander polynomial corresponding to $T_1(p,q)$, and hence the radius of stable terms is $1$. Recall in the proof of Prop.\ \ref{radius change if there were no stable term} we observed $l$ is odd in this case, which implies the length $l'$ for $T_1(p,q)$ is even. This in turn implies $T_1(p,q)$ corresponds to a two-component link, so $|\sigma(T_1(p,q))|\geq 1$ since it must be odd. Hence $\lfloor \frac{|\sigma|+1}{2} \rfloor \geq 1$.

\textbf{(Case 2) }If there are some stable terms in the Alexander polynomial corresponding to $(p,q)$. Let $i_0$, $i_0'$ be as in Prop.\ \ref{radius change when there are stable terms}. When $l=2n+1$, $|\sigma|=2k$. By assumption $k\geq \lfloor\frac{l-2(i_0-1)}{2}\rfloor=n-i_0+1$. If $i_0'=i_0+1$, then $\lfloor \frac{l'-2(i_0'-1)}{2}\rfloor=n-i_0+1\leq k$, and $\lfloor \frac{|\sigma'|+1}{2}\rfloor \geq \lfloor \frac{(2k-1)+1}{2}\rfloor=k$. If $i_0'=i_0$, then $\lfloor \frac{l'-2(i_0'-1)}{2}\rfloor=n-i_0+2\leq k+1$, and by Prop.\ \ref{radius change when there are stable terms}, Prop.\ \ref{signature is positive when half of b_i vanish} and Lemma \ref{effect of T_1 on signature},  $\lfloor \frac{|\sigma'|+1}{2}\rfloor \geq k+1$. When $l=2n$, the argument is similar and hence omitted.
\end{proof}
\bibliographystyle{abbrv}
\bibliography{HMconjbib}
\end{document}